\documentclass[british,11pt,dvipsnames,svgnames,x11names]{article}

\usepackage[british]{babel}
\usepackage{amsmath,amssymb,enumerate,amsthm}
\usepackage{cite}
\usepackage[latin1]{inputenc}
\usepackage[T1]{fontenc}
\usepackage{relsize}

\usepackage{psfig,labelfig}
\usepackage{graphicx}
\usepackage[dvips]{epsfig}
\usepackage[colorinlistoftodos]{todonotes}

\usepackage{lineno} 

\newtheorem{thm}{Theorem}[section]

\newtheorem{lem}[thm]{Lemma}

\newtheorem{question}{Question}

\theoremstyle{definition}
\newtheorem{defi}[thm]{Definition}
\newtheorem{rem}[thm]{Remark}

\newcommand{\K}{{\mathbb K}}

\newcommand{\Z}{{\mathbb Z}}

\DeclareMathOperator{\arccosh}{arccosh}

\title{Some counterexamples in surface homology}

\author{Peter Buser, Eran Makover and Bjoern Muetzel}

\begin{document}

\maketitle

\begin{abstract}
We present four counterexamples in surface homology. The first example shows that even if the loops inducing a homology basis intersect each other at most once, they still may separate the surface into two parts. The other three examples show some difficulties in working with minimal homology bases. Introducing hyperbolic ribbon graphs we modify the examples so as to have a hyperbolic metric.\\
\\
Keywords: surface topology, minimal homology basis.\\
\\
Mathematics Subject Classifications (2020): 14J80, 54G20 and 55N10.
\end{abstract}

\section{Introduction}\label{sec:Int}

While working with surface homology in \cite{bmm} we encountered a number of pitfalls and obstacles that easily lend themselves to be overlooked. The aim of this note is to clarify some of them for which we could not find a reference. We present our findings in four examples answering four questions. While the first question is of a more general nature the other three concern minimal homology bases. In the following $S$ stands for a closed orientable surface of genus $g \geq 2$.  We denote by $H_1(S;\K)$ its first homology group with coefficients in a ring $\K$, where we are mainly interested in the cases $\K=\Z$ and $\K=\Z_2$. The questions are about bases of $H_1(S;\K)$ that are induced by simple closed curves. When length problems are addressed we assume that $S$ is endowed with a length metric which must be such that in any homotopy class of closed curves there is one with minimal length. The length is denoted by $\ell$.

\bigskip

The first question is about configurations that have features of a canonical homology basis.

\begin{question}\label{qu1}Given a set of $2g$ simple closed  curves on $S$ that pairwise intersect each other at most once and, furthermore, induce a basis of $H_1(S;\Z)$, if we cut $S$ open along these curves, will the resulting surface $S'$ be connected?  
\end{question}

Though this is true in the case of a canonical homology basis, the answer is no in general. \textbf{Example 1} provides a configuration with five simple closed curves on a genus three surface whose complement is simply connected. By a general lemma (\textbf{Lemma \ref{lem:rank_n}}) the configuration can be completed into a homology basis and there is, in fact, an explicit one that provides a counterexample to \textbf{Question \ref{qu1}}. This example and the lemma are presented in Section \ref{sec:Sep}.

Section \ref{sec:Min} is about minimal homology bases. There are several concepts of short or minimal homology bases on surfaces that have a length metric. For instance, one may request that the longest member in the basis must have minimal length, or that the sum of the lengths shall be minimal. Here we focus on \emph{successive minimality}. We shall consider two variant definitions. The first is due to Gromov \cite{gr}.

\begin{defi}[Successive minima I]\label{def:minI} A set of curves $\alpha_1,\alpha_2, \ldots, \alpha_m$ on $S$ is called a set of successive minima (more precisely, of type I) for the homology $H_1(S;\K)$ if it is the result of the following search procedure:		
\begin{itemize}
\item[1)] take $\alpha_1$ to be a shortest homologically non-trivial curve;
\item[2)] for $j = 2,\ldots, m$, take $\alpha_j$ to be a shortest curve whose homology class $[\alpha_j]$ is not a linear combination over $\K$ of the homology classes of $\alpha_1, \ldots , \alpha_{j-1}$.
\end{itemize}
\end{defi} 

As Gromov observes in \cite[sect.\ 5.1]{gr} successive minima have the very useful property that they are simple and pairwise intersect at most once. Furthermore, they are \emph{straight} in the sense that for any two points $p$, $q$ on the curve the distance from $p$ to $q$ along the curve is the same as the global distance from $p$ to $q$ on the surface (see also \cite[Lemma 2.2]{gu}). Hence, 

\begin{question}\label{qu2} Does there always exist a basis of $H_1(S;\K)$ of successive minima of type I?  
\end{question}

If $\K$ is a field, e.g. $\K = \Z_2$, the answer is yes because there are always up to $m = 2g = \text{rank\,} H_1(S; \K)$ successive minima and, since $H_1(S; \K)$ is a vector space over $\K$ they then form a basis. In the general case, however, the answer is again no. Counterexamples are given in \textbf{Examples 2} and \textbf{3}. There we construct two surfaces $S_G$ and $S_H$ of genus two with a Riemannian metric of variable curvature and a surface $\hat{S}_G$ of genus twelve with a hyperbolic metric together with the respective loops.

This stands in contradiction with a statement in  \cite[sect.\ 5.1]{gr} at the bottom of p.\ 45, where the aim, however, only is to find \emph{generating} successive minima, and the latter can indeed be achieved by extending the search procedure beyond $m = \text{rank\,} H_1(S; \K)$ (see \textbf{Remark \ref{rem:Straight}}).

The examples also answer, in the affirmative, the following question raised in \cite[sect.\ 3.2]{bkp} that stands in connection with the difference between homology over a ring and homology over a field:

\begin{question}\label{qu3} Is it possible that $2g$ simple closed curves that pairwise intersect each other at most once form a basis of $H_1(S;\Z_2)$ but not so for $H_1(S;\Z)$?  
\end{question}

As the search of successive minima does not always lead to a basis unless $\K$ is a field one may resort to the following alternative which remedies this:

\begin{defi}[Successive minima II]\label{def:minII} A set of curves $\alpha_1,\alpha_2, \ldots, \alpha_m$ on $S$ is called a set of successive minima of type II for the homology $H_1(S;\K)$ if it is the result of the following search procedure:					
\begin{itemize}
\item[1)] take $\alpha_1$ to be a shortest homologically non-trivial closed curve;
\item[2)] for $j = 2, \ldots , m$, take $\alpha_j$ to be a shortest closed curve with the property that the homology classes of $\alpha_1, \ldots , \alpha_{j-1}, \alpha_j$ can be extended to a basis of $H_1(S;\K)$.
\end{itemize}
\end{defi}
It is easily checked that whenever procedure I (with $m = 2g$) does lead to a basis then procedures I and II come to the same result. In particular, when $\K$ is a field the two definitions of successive minima coincide. 

One may now call a basis \emph{minimal} if it is a sequence of successive minima of type II. But ``how minimal'' is it from other points of view? In \cite[p.\ 1059]{gu} Guth defines the following partial order on the set of all bases of $H_1(S;\K)$: Let $A=\{[\alpha_1], \dots, [\alpha_{2g}]\}$, $B=\{[\beta_1], \dots, [\beta_{2g}]\}$ be bases and assume that the representing curves are of minimal lengths in their homology classes and that the sequences are ordered by length: $\ell(\alpha_1) \leq \dots \leq \ell(\alpha_{2g})$, $\ell(\beta_1) \leq \dots \leq \ell(\beta_{2g})$. We say that $A$ is \emph{smaller or equal} $B$ if
 \begin{equation*}
\ell(\alpha_k) \leq \ell(\beta_k), \quad \text{for all $k = 1, \dots, 2g$},
 \end{equation*}
and $A$ is \emph{strictly smaller} than $B$ if, in addition, $\ell(\alpha_j) < \ell(\beta_j)$ for some $j$. Based on this partial order we call $A$ \emph{locally minimal} if there is no strictly smaller basis and \emph{globally minimal} if it is smaller or equal $B$ for \emph{any} other basis $B$. Of course, this prompts

\begin{question}\label{qu4}Does a globally minimal basis always exist?  
\end{question}

It is not difficult to prove (see Section \ref{sec:Clos}) that if successive minima procedure I leads to a basis then this basis must be globally minimal. Hence, the answer to \textbf{Question \ref{qu4}} is yes if $\K$ is a field; unfortunately, in the general case it is once again no. A counterexample on a genus four surface is given in \textbf{Example 4}. In addition, the example can be modified so that it is hyperbolic, albeit with a greater genus. This is carried out in  \textbf{Remark \ref{rem:HypExpl}} respectively, \textbf{Example 3}, where we introduce a general construction of hyperbolic ribbon graphs.

A word about how we arrived at the configurations shown in \textit{Figures} 
\ref{fig:genus3}, \ref{fig:configurationsGH} and \ref{fig:G4}. They gradually emerged after many futile attempts to prove the opposite and by modifying numerous sketches through trial and error. It would be interesting to find a technique that makes it possible to construct further examples in a systematic way.

\section{Separating homology bases}\label{sec:Sep}

This section is about separation properties of homology bases on surfaces. \textbf{Example 1} will show that even if the curves of a basis are simple and pairwise intersect each other at most once their union set may be separating the surface. To put this example into perspective we first prove a lemma that goes in the other direction saying that as long as a system of simple closed curves is \emph{not} separating the system is always part of a homology basis. This is even so in cases where, like in \textit{Figure \ref{fig:genus3}}, the complement of the curve system is a simply connected domain. We could not find a reference for the lemma in the literature and we include a proof for the convenience of the reader.

In the following we call a \emph{partial basis} a set of homology classes that can be completed into a homology basis.

\begin{lem}\label{lem:rank_n}
Let $S$ be a closed orientable surface of genus $g \geq 1$ and let $a_1, \ldots, a_{n}$   be distinct curves on $S$ with the following properties
\begin{enumerate}
\item[1.)] $a_1, \ldots, a_{n}$ are simple closed curves,
\item[2.)] $S \smallsetminus (a_1 \cup \ldots \cup a_{n})$ is connected.
\end{enumerate}
Then $n \leq 2g$ and the homology classes $[a_1], \ldots, [a_{n}]$ form a partial basis of $H_1(S;\K)$.
\end{lem}

We sketch a proof. All arguments are standard.

\begin{proof} 

We assume $\K = \Z$ and argue for general $\K$ at the end. Without loss of generality we may assume that the system of curves is \emph{maximal} in the sense that it cannot be extended to a system $a_1, \ldots, a_{n}, a_{n+1}$ satisfying 1.) and 2.). Then, cutting $S$ open along $a_1, \ldots, a_{n}$ we obtain a topological sphere $S'$ with $m$ holes (i.e.\ a surface of signature $(0;m)$) for some $m \geq 1$. In turn, $S$ is the quotient  
\begin{equation*}
S = S'/(\text{mod pasting})
\end{equation*}
``pasting'' meaning the reverse of the cutting. Note that any pair of points $p_1, p_2$ on the boundary $\partial S'$ of $S'$ that come together in this pasting must belong to the \emph{same} connected component of $\partial S'$ for otherwise we could draw a simple curve $a_{n + 1}$ from $p_1$ to $p_2$ on $S'$ and $a_1, \ldots, a_{n}, a_{n+1}$ on $S$ would still satisfy 1.) and 2.) contradicting the maximality.

The argument is via graphs. For this the curves are homotoped so that they pairwise intersect in at most finitely many points and $S$ is viewed as polyhedral surface on which $a_1, \ldots, a_{n}$ \emph{are edge paths that pairwise intersect only in vertices}. Note that this may be carried out such that conditions 1.) and 2.) are still satisfied.

For any connected component $c$ of $\partial S'$  we let $F_c$ be a small tubular neighbourhood of $c$ in $S'$, ``small'' meaning that all $F_c$ are annuli and pairwise disjoint. For each $c$ we set 
\begin{equation*}
S_c = F_c/(\text{mod pasting}).
\end{equation*}
This represents $S$ as an $m$-holed sphere $S^0$ to which $m$ surfaces $S_c, S_{c'}, \ldots$ are attached along the holes. Since the latter are orientable the first homology group of $S$ is the direct sum of the first homology groups of $S_c, S_{c'}, \ldots$ (to see this introduce canonical generators on the latter). Therefore we can inspect each $S_c$ individually and hence, we may assume without loss of generality that $m=1$.   Let
\begin{equation*}
\mathcal{G} = c/(\text{mod pasting})
\end{equation*}
be the connected graph obtained by restricting the pasting to $c$. As a point set  $\mathcal{G}$ is the same as $a_1 \cup \ldots \cup a_{n}$. Furthermore, $a_1, \ldots, a_{n}$ are naturally identified as cycles of $\mathcal{G}$. 

Since $\mathcal{G}$ is a deformation retract of $S_c$ there is a natural isomorphism $\mathbf{j} : H_1(S_c;\Z) \to H_1(\mathcal{G};\Z)$ that acts as the identity on the homology classes of $a_1 , \ldots,  a_{n}$. 

It remains to construct a basis of $H_1(\mathcal{G};\Z)$ via a spanning tree e.g. \cite[sect.\ 2.1.5]{st}. For this we delete from each cycle $a_k$  some small arc $e_k$, $k=1,\ldots,n$, ``small'' meaning that the arc lies on some edge of the graph and does not meet the endpoints of this edge.  As the $a_k$ intersect only in vertices the $e_k$ are pairwise distinct and the resulting graph $\mathcal{G}'$ is connected. If $\mathcal{G}'$ still has cycles we delete successively  further small arcs $f_1, \ldots f_q$ belonging to cycles $b_1, \ldots, b_q$ until the resulting graph $\mathcal{G}''$ is a tree. Then the homology classes of $a_1, \ldots, a_{n}, b_1, \ldots, b_q$   form a basis of $H_1(\mathcal{G};\Z)$. Via the isomorphism $\mathbf{j}$ it is identified with a basis of $H_1(S_c;\Z)$ and $n +q = 2g$. Finally, $a_1, \ldots, a_{n}, b_1, \ldots, b_q$ also induce a basis of $H_1(S;\Z)$ given that $S$ is orientable and $S_c$ is obtained from $S$ by removing a disk.

This proves the lemma for $\K = \Z$, and for the general case it then holds as well because for orientable surfaces one has $H_1(S;\K) = H_1(S;\Z) \otimes \K$.
\end{proof}

\begin{figure}[!ht]
\SetLabels
(.19*.90) $S_E$\\
(.27*.77) $\beta_1$\\
(.27*.77) $\beta_1$\\
(.25*.27) $\beta_2$\\
(.43*.30) $\,\beta_3$\\
(.20*.18) $\gamma$\\
(.41*.44) $\delta$\\
(.69*.83) $1$\\
(.69*.90) $\delta$\\
(.61*.73) $1'$\\
(.59*.78) $\delta$\\
(.64*.80) $\,\,\,2$\\
(.63*.87) $\beta_1$\\
(.58*.30) $2'$\\
(.55*.30) $\beta_1$\\
(.58*.64) $3$\\
(.55*.67) $\gamma$\\
(.68*.13) $\,\,\,3'$\\
(.68*.07) $\,\,\,\gamma$\\
(.56*.55) $\,\,\,4$\\
(.53*.55) $\beta_3$\\
(.79*.80) $4'$\\
(.80*.87) $\,\,\,\beta_3$\\
(.56*.41) $\,\,\,5$\\
(.53*.41) $\,\,\delta$\\
(.61*.23) $5'$\\
(.58*.18) $\,\,\,\delta$\\
(.64*.16) $\,\,6$\\
(.63*.11) $\beta_3$\\
(.87*.41) $\,\,6'$\\
(.90*.41) $\,\,\,\beta_3$\\
(.75*.13) $7$\\
(.75*.06) $\delta$\\
(.83*.24) $7'$\\
(.84*.18) $\,\,\,\delta$\\
(.78*.16) $\,\,\,\,8$\\
(.80*.10) $\,\,\beta_2$\\
(.86*.64) $8'$\\
(.89*.67) $\,\beta_2$\\
(.85*.32) $\,\,\,\,9$\\
(.88*.31) $\,\,\,\gamma$\\
(.75*.83) $\,9'$\\
(.75*.91) $\gamma$\\
(.87*.55) $\,\,10$\\
(.90*.54) $\,\,\delta$\\
(.82*.73) $\,\,10'$\\
(.85*.78) $\delta$\\
\endSetLabels
\AffixLabels{%
\centerline{%
\includegraphics[height=7cm,width=14cm]{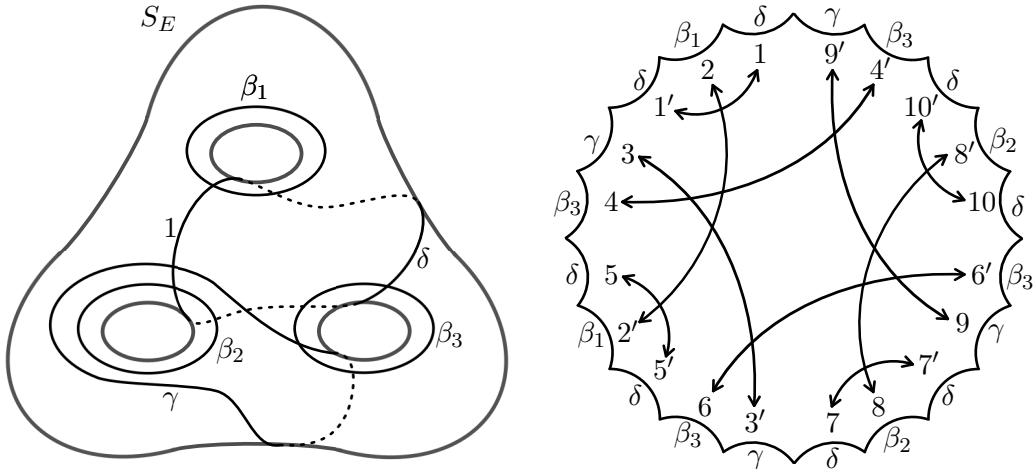}
}}
\vspace{-12pt}
\caption{The genus 3 surface $S_E$ cut open into a non canonical fundamental polygon along $\beta_1, \beta_2, \beta_3, \gamma$ and $\delta$.}
\label{fig:genus3}
\end{figure}

A special situation for the preceding lemma is given when the complement of the curves is simply connected. One may ask oneself whether under such a circumstance the curve system actually \emph{is} already a homology basis. The question is closely related to \textbf{Question \ref{qu1}} and the next example is a counterexample to both.

\bigskip
\textbf{Example 1.} Let $S_E$ be the surface of genus 3 depicted in \textit{Figure \ref{fig:genus3}}. Let $\beta_1, \beta_2, \beta_3,\gamma ,\delta$ be the curves shown there. The complement of the five curves $\beta_1, \beta_2, \beta_3,\gamma,\delta$ is a polygon domain with consecutive sides 
\[
1, 2, 1', 3, 4, 5, 2', 5', 6, 3', 7, 8, 7', 9, 6', 10, 8', 10', 4', 9'
\]
and the surface is obtained from it by pasting sides $k$ and $k'$ together for $k =1,\ldots,10$. 
Hence, the five curves are simple, they pairwise intersect each other in at most one point and their complement is simply connected.

We thus have an instance of \textbf{Lemma \ref{lem:rank_n}} stating that we can complete $\beta_1, \beta_2, \beta_3,\gamma,\delta$ into a homology basis. Explicitly, we take a simple closed curve $\alpha_1$ that intersects $\beta_1$ once and is disjoint from the other four. We let its orientation be such that the couple $\alpha_1$, $\beta_1$ is positively oriented. On the fundamental domain on the right hand side $\alpha_1$ then appears as simple arc that goes from side $2$ to side $2'$. To see that this indeed yields a basis we note that with respect to the canonical homology basis $\alpha_1, \beta_1, \alpha_2, \beta_2, \alpha_3, \beta_3$ (the $\alpha_i$ are not shown in \textit{Figure \ref{fig:genus3}}, they are analogous to those in \textit{Figure \ref{fig:G4}}) the curves $\gamma$, $\delta$ are given as follows:
\begin{equation*}
\gamma = \beta_2 + \alpha_3, \quad \delta = \alpha_1 - \alpha_2 + \alpha_3,
\end{equation*}
which in turn yields
\begin{equation*}
\alpha_2 = \alpha_1 - \beta_2 + \gamma - \delta, \quad \alpha_3 = \gamma-\beta_2.
\end{equation*}
The curves in the basis $\alpha_1, \beta_1, \beta_2, \beta_3,\gamma,\delta$ are simple, pairwise intersect at most once and separate the surface into two parts.

We would like to add that in genus 2 a similar example does not exist, see \textbf{Remark \ref{rem:nearcang2}}.

\section{Minimal homology bases}\label{sec:Min}

In the first part of this section we answer to \textbf{Questions \ref{qu2}} and \textbf{\ref{qu3}}. \textbf{Example 2} provides two surfaces of genus two with variable curvature, \textbf{Example 3} is an adaption of the first surface in genus twelve that has a hyperbolic metric.

\begin{figure}[!h]

\vspace{-6pt}

\SetLabels
(.01*.48) $1$\\
(.16*.48) $1'$\\
(.08*.42) $\,\alpha$\\
(.18*.62) $\beta$\\
(.25*.79) $\,\gamma$\\
(.20*.48) $2$\\
(.29*.48) $\,3$\\
(.27*.41) $\,\delta$\\
(.38*.48) $2'$\\
(.48*.48) $\,3'$\\
(.45*.41) $\,\,\delta$\\
(.25*.01) $\mathcal{G}$\\
(.52*.48) $1$\\
(.70*.48) $1'$\\
(.60*.42) $\,\alpha$\\
(.71*.62) $\beta$\\
(.77*.79) $\gamma$\\
(.59*.60) $2$\\
(.69*.37) $3$\\
(.62*.62) $\,\,\,\delta$\\
(.89*.48) $\,3'$\\
(.99*.48) $\,2'$\\
(.97*.41) $\delta$\\
(.75*.01) $\mathcal{H}$\\
\endSetLabels
\AffixLabels{%
\centerline{%
\includegraphics[height=5cm,width=15.5cm]{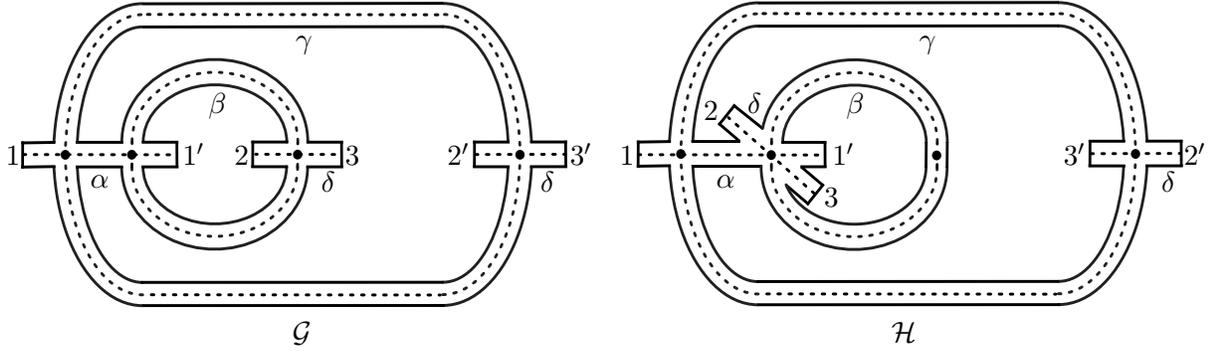}}}
\caption{The ribbon graphs $\mathcal{G}$ and $\mathcal{H}$. Side $k$ is pasted to side $k'$ for $k=1, 2, 3$.}
\label{fig:flatribbon}
\end{figure}

\textbf{Example 2.} We construct the examples using the ribbon graphs $\mathcal{G}$, $\mathcal{H}$ obtained from the plane domains shown in \textit{Figure \ref{fig:flatribbon}} by pasting sides $k$ and $k'$ together for $k = 1,2,3$ in an orientable way. Topologically these are orientable surfaces  of genus 2 with two boundary components. The four simple closed curves (dashed) $\alpha, \beta, \gamma, \delta$ on $\mathcal{G}$ respectively, $\mathcal{H}$ form the soul with $\alpha$ going from side $1$ to side $1'$, $\delta$ from $2$ via $3, 3'$ to $2'$, $\beta$ running along the inner and $\gamma$ along the outer circuit. In $\mathcal{H}$ we have added a ``dummy'' vertex to $\beta$ so that here too $\beta$ has two edges. By attaching topological disks along the boundary components of $\mathcal{G}$ and $\mathcal{H}$ we obtain closed surfaces $S_G$ and $S_H$ of genus 2 marked with the homotopy (not homology) classes of $\alpha, \beta, \gamma, \delta$. \textit{Figure \ref{fig:configurationsGH}} depicts these curves on $S_G$ and $S_H$ in another form. 

\begin{figure}[!b]
\vspace{-4pt}
\SetLabels
(.31*.67) $\alpha_1$\\
(.44*.81) $\beta_1$\\
(.68*.67) $\,\alpha_2$\\
(.55*.81) $\beta_2$\\
(.05*.18) $\alpha$\\
(.17*.33) $\,\beta$\\
(.29*.36) $\gamma$\\
(.26*.25) $\,\,\delta_G$\\
(.57*.18) $\alpha$\\
(.71*.30) $\,\beta$\\
(.81*.36) $\gamma$\\
(.77*.29) $\delta_H$\\
(.24*.02) $S_G$\\
(.76*.02) $S_H$\\
\endSetLabels
\AffixLabels{%
\centerline{%
\includegraphics[height=7.2cm,width=15cm]{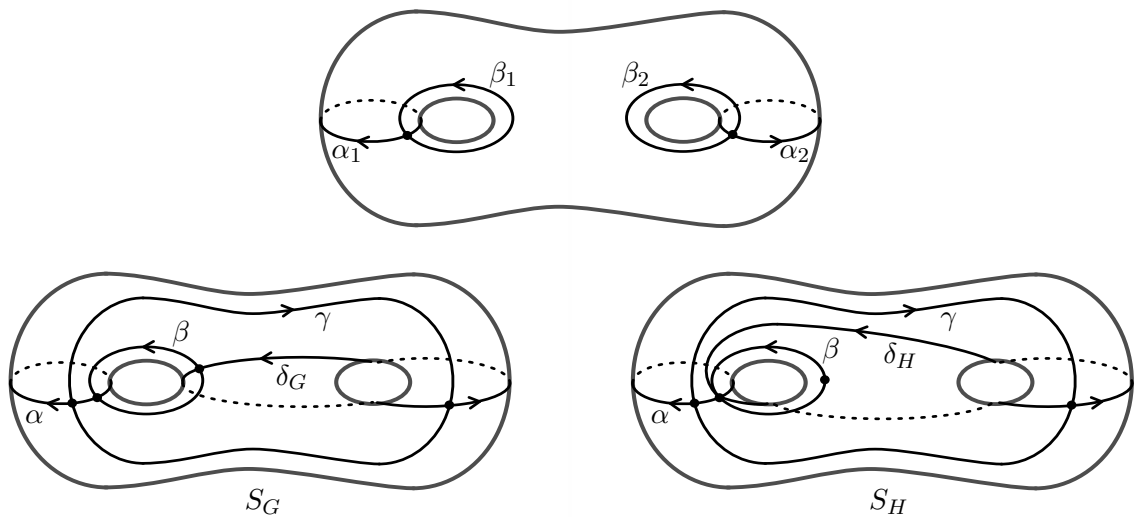}
}}
\caption{The two surfaces $S_G$ and $S_H$ of genus 2 with the configurations from $\mathcal{G}$ and $\mathcal{H}$, respectively. }
\label{fig:configurationsGH}
\end{figure}

The two configurations are different in the sense that there is no homeomorphism from $S_G$ to $S_H$ that sends the homotopy classes of the four curves on $S_G$ to those on $S_H$ 
(on $S_H$ there is a triple of curves with pairwise algebraic intersection number 1, on $S_G$ there is no such triple). Hence, we have two slightly different examples.

We now verify that $\alpha, \beta, \gamma, \delta$ are linearly independent but not generators in the $\Z$ homology (confounding curves with their homology classes). For this we use the canonical homology basis $\alpha_1, \beta_1, \alpha_2, \beta_2$ shown on  top in \textit{Figure \ref{fig:configurationsGH}}. We get the following where we write $\delta$ as $\delta_G$ and $\delta_H$, for distinction:

 \begin{equation*}
 \begin{alignedat}{2}
\rule{5.4em}{0pt}&\rule{10em}{0pt}&\rule{6em}{0pt}&\rule{11em}{0pt}\\[-20pt]
\text{on $S_G$:}\rule{2em}{0pt}\alpha &= \alpha_1 & \text{on $S_H$:}\rule{2em}{0pt}\alpha &= \alpha_1 \\
\beta &= \beta_1 & \beta &= \beta_1\\
\gamma &= -\beta_1 -\beta_2 & \gamma &= -\beta_1 -\beta_2\\
\delta_G &= -\alpha_1+2\alpha_2+\beta_2 & \delta_H &= -\alpha_1+\beta_1+2\alpha_2+\beta_2.
\end{alignedat}
 \end{equation*}%

In return this yields 

 \begin{equation*}
 \begin{alignedat}{2}
 \rule{5.4em}{0pt}&\rule{10em}{0pt}&\rule{6em}{0pt}&\rule{11em}{0pt}\\[-20pt]
\alpha_2 &= \tfrac{1}{2}(\alpha+\beta+\gamma + \delta_G) & \alpha_2 &= \tfrac{1}{2}(\alpha+\gamma+\delta_H)\\
\beta_2 &= -\beta-\gamma & \beta_2 &= -\beta-\gamma.\\
\end{alignedat}
 \end{equation*}%

We deduce from this that in both cases the four curves generate a subgroup of the homology group of index 2.

It remains to define a metric on $S_G$ and $S_H$ such that $\alpha, \beta, \gamma, \delta$ are successive minima in the sense of \textbf{Definition \ref{def:minI}}. For this we observe that in the combinatorial graphs formed by the curves $\alpha, \beta, \gamma, \delta$ (the souls of $\mathcal{G}$ and $\mathcal{H}$) these curves  seen as edge paths all have combinatorial length 2 while any other non-trivial closed edge path has length greater or equal 3. We now take flat Riemannian metrics on $\mathcal{G}$ and $\mathcal{H}$ such that the curves become closed geodesics of length 2 and the boundary curves are at small distance (one may think of pasting together paper bands). With these metrics on $\mathcal{G}$ and $\mathcal{H}$ the curves $\alpha, \beta, \gamma, \delta$ are systoles and the same holds for them on $S_G$ and $S_H$ 
 if we take for each boundary component $b$ of $\mathcal{G}$ and $\mathcal{H}$ a Riemannian metric on the attached disk $D_b$ such that the boundary of $D_b$ has the same length as $b$ and is straight (see \textbf{Remark \ref{rem:Straight}}) in $D_b$.
 
We point out that, more precisely, $\alpha, \beta, \gamma, \delta$ are the unique systoles. Hence, any sequence of successive minima is a permutation of $\alpha, \beta, \gamma, \delta$.
\\

\begin{figure}[!t]
\SetLabels
(.49*.95) $1$\\
(.49*.01) $\,1'$\\
(.48*.60) $\alpha$\\
(.29*.45) $2'$\\
(.29*.95) $3'$\\
(.27*.77) $\,\,\delta$\\
(.17*.70) $4'$\\
(.61*.70) $4$\\
(.35*.68) $\gamma$\\
(.38*.26) $5$\\
(.81*.26) $\,\,\,5'$\\
(.67*.23) $\beta$\\
(.70*.01) $2$\\
(.70*.50) $3$\\
(.69*.37) $\delta$\\
(.48*.37) $\,s$\\
(.55*.24) $s$\\
(.56*.38) $t$\\
(.27*.23) $\tilde{S}_G$\\
\endSetLabels
\AffixLabels{%
\centerline{%
\includegraphics[height=7.5cm,width=10.7cm]{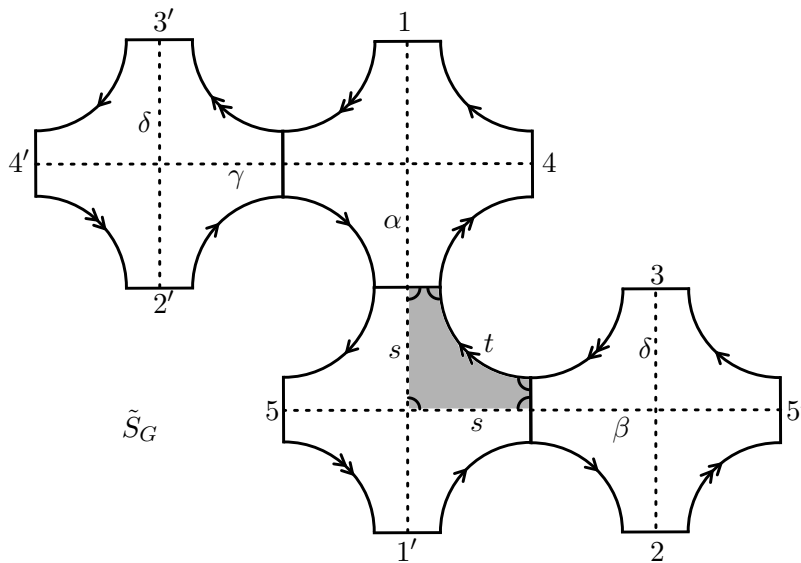}
}}
\caption{Four octagons pasted together forming a hyperbolic surface $\tilde{S}_G$ of genus 2 with two geodesic boundary components of length $8t$.}
\label{fig:octagons}
\end{figure}

\textbf{Example 3} It is also possible to carry out the construction with a hyperbolic metric, however with higher genus. We carry this out for the first configuration.

In the first place we put a hyperbolic metric on the ribbon graph $\mathcal{G}$ such that the two boundary components are closed geodesics. To this end we assign to each vertex of the graph a hyperbolic geodesic octagon as shown in \textit{Figure \ref{fig:octagons}}. It consists of four identical right angled geodesic pentagons (gray shaded area in the figure) with two adjacent sides of length $s$ to be determined below and opposite side of length $t$. By hyperbolic trigonometry

 \begin{equation*}\label{Eq:pentagon}
\cosh(t) = \sinh^2(s).
 \end{equation*}

For our purposes it is practical to choose $s$ such that $s=t$. This leads to the equation $\sinh^2(s) = \cosh(s)$ with $\cosh(s)$ a solution to the equation $y^2-1 = y$. The approximate value is

 \begin{equation*}\label{Eq:sequalt}%
s = t = 1.061 \ldots .
 \end{equation*}%

We paste four copies of the octagon together by the pattern shown in the figure. This pattern corresponds to the one in \textit{Figure \ref{fig:flatribbon}} and so from the topological point of view the resulting surface $\tilde{S}_G$ is identical with the ribbon graph $\mathcal{G}$. (To see this add in \textit{Figure \ref{fig:flatribbon}} two additional cuts on graph $\mathcal{G}$ corresponding to 4 and 5 so as to get the same pasting pattern of octagons as in \textit{Figure \ref{fig:octagons}}.)

We now endow $\tilde{S}_G$ with the hyperbolic metric inherited from the octagons. In this metric the two boundary components are closed geodesics of length $8t = 8s$ and the four curves $\alpha, \beta, \gamma, \delta$ are closed geodesics each of length $4s$. Furthermore, any non-trivial closed curve on $\tilde{S}_G$ not homotopic to $\alpha, \beta, \gamma, \delta$ (or their inverses) must cross at least four octagons and therefore has length bigger than $4t = 4s$. Hence, here too $\alpha, \beta, \gamma, \delta$ are the systoles and therefore the first successive minima.

\begin{figure}[!b]
\vspace{0pt}
\SetLabels
(.15*.75) $u_k$\\
(.25*.75) $v_k$\\
(.35*.75) $u_k'$\\
(.44*.75) $\,\,v_k'$\\
(.20*.55) $a_k$\\
(.41*.55) $\,\,b_k$\\
(.39*.71) $\,\,\,\vartheta$\\
(.59*.71) $\vartheta$\\
(.43*.33) $\,\,w$\\
(.56*.33) $w$\\
(.48*.02) $\tau$\\
(.52*.02) $\tau$\\
\endSetLabels
\AffixLabels{%
\centerline{%
\includegraphics[height=4.7cm,width=14cm]{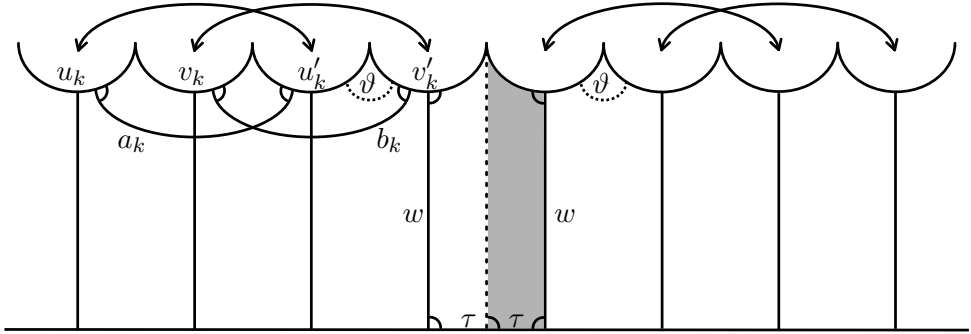}
}}
\caption{The surface $W$ of signature $(h;1)$ is a hyperbolic annulus with alternating side \newline identification.}
\label{fig:crown}
\end{figure}

To create a closed hyperbolic surface $\hat{S}_G$ out of $\tilde{S}_G$ where this property persists we attach a suitable hyperbolic surface $W$ of some signature $(h;1)$ along each of the two boundary geodesics of $\tilde{S}_G$ where the boundary of $W$ is a closed geodesic of the same length $8t$ and where the geometry must be such that the boundary collar has width $w > 2s$.

We construct $W$ as follows, beginning with a hyperbolic annulus or ``crown'' on which one boundary is a closed geodesic of length $8t$ and the other boundary is a regular polygon curve with $20$ sides of equal length whose interior angles at the vertices all are equal to

\begin{equation*}\label{Eq:delta}%
\vartheta = \frac{\pi}{10} \text{ \ \ (see \textit{Figure \ref{fig:crown}}).} 
 \end{equation*}

The crown consists of 40 trirectangles with acute angle $\vartheta/2$ and opposite sides of lengths $w$, $\tau$, where we choose $\tau = 8t/40$ (gray shaded area in \textit{Figure \ref{fig:crown}}). The width $w$ is determined by the formula
 \begin{equation*}
\sinh(w) \sinh(\tau) =\cos(\vartheta/2).
 \end{equation*}
Numerically we get
 \begin{equation*}
w = 2.234 \ldots .
 \end{equation*}
which is greater than $2t$, as requested. We now pairwise identify the sides of the 20-gon via alternating sides identification in the same way as the sides of a canonical fundamental polygon of an orientable surface are pasted together: the sides are consecutively labelled $\ldots, u_k, v_k, u'_k, v'_k, \ldots$ for $k=1,\ldots,5$, and then for each $k$ side $u_k$ is pasted to $u'_k$ and $v_k$ to $v'_k$. All vertices are pasted together to the same point and since the angle sum is $20 \vartheta = 2\pi$ the result is a hyperbolic surface $W$ of signature $(5;1)$ with a boundary geodesic of length $8t$ and a boundary collar of width $w$.

For any $k$ there is a common perpendicular $a_k$ from side $u_k$ of the crown to side $u'_k$, and there is a similar perpendicular $b_k$ from $v_k$ to $v'_k$. By the symmetries of the crown their endpoints come together in the pasting and the $a_k$, $b_k$ become smooth closed geodesics on $W$. As in the case of a canonical fundamental polygon $a_1, b_1, \ldots, a_5, b_5$ form a canonical homology basis on $W$.

To compute the lengths of these geodesics we observe (see \textit{Figure \ref{fig:crown}}) that for each $a_k$ there is a right angled pentagon with side of length $\ell(a_k)/2$ and opposite sides  of length $2\tau$ and $w$. The pentagon formula gives
\begin{equation*}
\cosh(\tfrac{1}{2} \ell(a_k)) = \sinh(2 \tau) \sinh(w)
\end{equation*}
and the same holds for $b_k$. Numerically we get
\begin{equation*}
\ell(a_k) = \ell(b_k) =2.656 \ldots .
\end{equation*}
From the pasting pattern of the crown it follows that any simple closed curve on $W$ that is not homotopic to a point or homotopic to the boundary must contain at least two arcs that connect one of the sides of the 20-gon with one of the 19 non adjacent perpendiculars of length $w$. The shortest such connectors are the halves of the $a_k$ and $b_k$. Consequently, the $a_k$ and $b_k$ are the systoles of $W$.

We now paste two copies of $W$ along the boundary geodesics of $\tilde{S}_G$ to get a closed hyperbolic surface $\hat{S}_G$ of genus 12. Since the systoles $a_k, b_k$ on the two copies of $W$ are shorter than the systoles $\alpha, \beta, \gamma, \delta$ on $\tilde{S}_G$ and since the boundary collar of $W$ is larger than half the lengths of the latter it follows that the search of successive minima first outputs the minimal homology bases on the copies of $W$ and then the four systoles of $\tilde{S}_G$ at which moment the rank of $H_1(\hat{S}_G;\Z)$ is reached and the procedure halts. The curves thus found generate, like in \textbf{Example 2}, a subgroup of $H_1(\hat{S}_G;\Z)$ of index 2. 

Over $\Z_2$ or any other field the curves in the three examples form homology bases and the examples thus say yes to \textbf{Question \ref{qu3}}.

\bigskip

In the second part of this section we deal with \textbf{Question \ref{qu4}} that came up in connection with the strength of the search procedure in \textbf{Definition \ref{def:minII}}. Does it produce or, asking more generally, must there always exist a basis $\alpha_1, \dots, \alpha_{2g}$, ordered by increasing length which is minimal to the point that for any other homology basis $\beta_1, \dots, \beta_{2g}$ ordered by increasing length we have $\ell(\alpha_k) \leq \ell(\beta_k)$, $k=1,\dots, 2g$. The next example shows that this request for minimality is too strong, in general.

\begin{figure}[!ht]
\vspace{12pt}
\SetLabels
(.17*.39) $\gamma$\\
(.20*.58) $\gamma$\\
(.13*.14) $\delta$\\
(.23*.67) $\delta$\\
(.28*.39)  $\,\,\,\delta$\\
(.14*.50) $\delta$\\
(.37*.95) $\,\,S_K$\\
(.37*.43) $\alpha_1$\\
(.62*.43) $\,\,\,\alpha_2$\\
(.52*.90) $\,\,\alpha_3$\\
(.47*.10) $\alpha_4$\\
(.45*.50) $\,\,\,\beta_1$\\
(.53*.50) $\,\beta_2$\\
(.44*.75) $\,\,\,\beta_3$\\
(.54*.25) $\,\,\beta_4$\\
(.76*.70) $\lambda$\\
(.86*.38) $\lambda$\\
(.84*.50) $\mu$\\
(.94*.41) $\mu$\\
\endSetLabels
\AffixLabels{%
\centerline{%
\includegraphics[height=5.5cm,width=15.5cm]{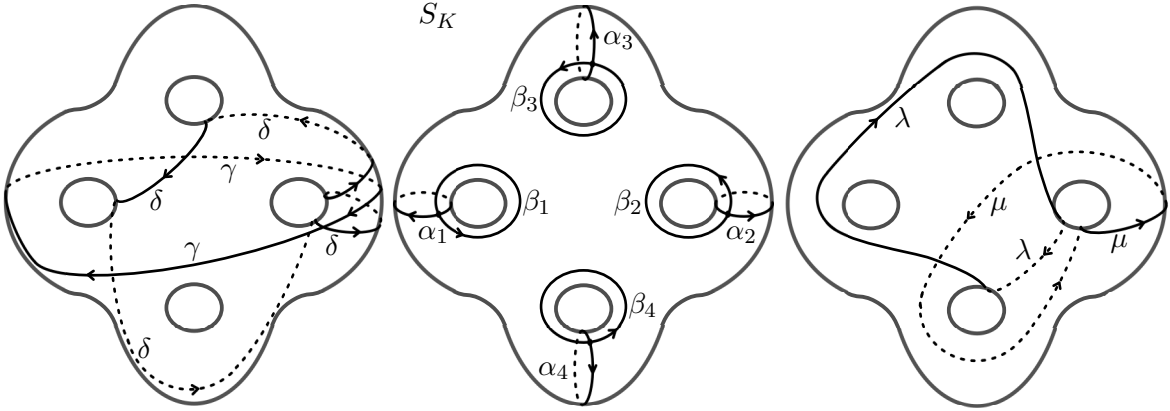}
}}
\vspace{12pt}
\caption{Curves on the genus 4 surface $S_K$. }
\label{fig:G4separately}
\end{figure}


\textbf{Example 4.} The idea is similar to the preceding ones but now we take a surface $S_K$ of genus 4. The configuration on $S_K$ is an enhanced form of the configuration in $S_G$ as of \textit{Figure \ref{fig:configurationsGH}} consisting of 10 closed curves and is shown in \textit{Figure \ref{fig:G4}}. For better understanding its constituents are represented separately in \textit{Figure \ref{fig:G4separately}}.

The curves $\alpha_1, \beta_1, \ldots, \alpha_4, \beta_4$ depicted in the middle form a canonical basis of $H_1(S_K; \Z)$; of these the pair $\alpha_2$, $\beta_2$ is not used in the configuration. With respect to this basis the curves $\gamma, \delta, \lambda, \mu$ have the following representations, where we identify curves with their homology classes.

\begin{align*}
\gamma &= -\beta_1 - \beta_2\\
\delta &= -\alpha_1 + 2 \alpha_2 + \beta_2 + \alpha_3 + \beta_4\\
\lambda &= -\beta_1 -\alpha_2 -\beta_3 + \alpha_4\\
\mu &= \alpha_2 + \beta_2 + \beta_4
\end{align*}

For better overview we rename the curves $\alpha_1, \gamma, \delta, \alpha_3, \beta_3, \alpha_4, \beta_4, \beta_1, \lambda, \mu$ in this order $u_1, \ldots , u_{10}$. They then have the following coordinates with respect to the basis vectors $\alpha_1, \beta_1, \ldots, \alpha_4, \beta_4$:
 \begin{equation*}
\begin{array}{cccccrrrrrrrr}
u_1&{\!\!\!=\!\!\!}&\alpha_1&\!\!\!: & \rule{1em}{0pt}& 1&0&0&0&\phantom{-}0&0&\phantom{-}0&\phantom{-}0\\
u_2&{\!\!\!=\!\!\!}&\gamma&\!\!\!: & \rule{1em}{0pt}& 0&-1&0&-1&0&0&0&0\\
u_3&{\!\!\!=\!\!\!}&\delta&\!\!\!: & \rule{1em}{0pt}& -1&0&2&1&1&0&0&1\\
u_4&{\!\!\!=\!\!\!}&\alpha_3&\!\!\!: & \rule{1em}{0pt}& 0&0&0&0&1&0&0&0\\
u_5&{\!\!\!=\!\!\!}&\beta_3&\!\!\!: & \rule{1em}{0pt}& 0&0&0&0&0&1&0&0\\
u_6&{\!\!\!=\!\!\!}&\alpha_4&\!\!\!: & \rule{1em}{0pt}& 0&0&0&0&0&0&1&0\\
u_7&{\!\!\!=\!\!\!}&\beta_4&\!\!\!: & \rule{1em}{0pt}& 0&0&0&0&0&0&0&1\\
u_8&{\!\!\!=\!\!\!}&\beta_1&\!\!\!: & \rule{1em}{0pt}& 0&1&0&0&0&0&0&0\\
u_9&{\!\!\!=\!\!\!}&\lambda&\!\!\!: & \rule{1em}{0pt}& 0&-1&-1&0&0&-1&1&0\\
u_{10}&{\!\!\!=\!\!\!}&\mu&\!\!\!: & \rule{1em}{0pt}& 0&0&1&1&0&0&0&1
\end{array}
 \end{equation*}
%
\begin{figure}[!t]
\SetLabels
(.36*.95) $S_K$\\
(.17*.49) $\,\,\alpha_1$\\
(.53*.89) $\,\,\alpha_3$\\
(.46*.10) $\,\,\alpha_4$\\
(.40*.57) $\beta_1$\\
(.42*.71) $\beta_3$\\
(.58*.24) $\beta_4$\\
(.36*.70) $\lambda$\\
(.52*.50) $\mu$\\
(.44*.52) $\,\,\,\delta$\\
(.67*.33) $\,\,\,\delta$\\
(.53*.41) $\gamma$\\
\endSetLabels
\AffixLabels{%
\centerline{%
\includegraphics[height=8cm,width=10.3cm]{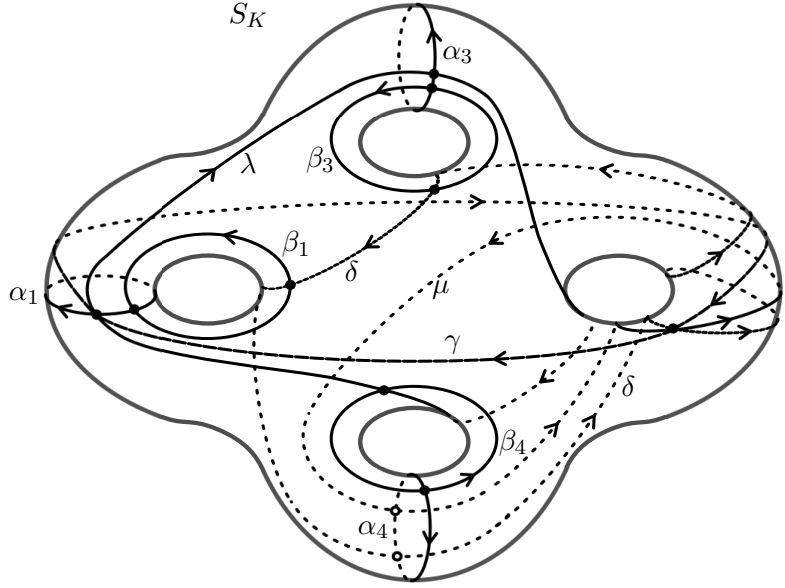}
}}
\caption{Graph $\mathcal{K}$ on $S_K$ formed by the curves $u_1, \ldots, u_{10}$.}
\label{fig:G4}
\end{figure}

To see which combinations of $u_1, \ldots, u_{10}$ form bases we shall denote by $\text{det}[u_{i_1}, \ldots, u_{i_8}]$ the determinant of the matrix formed by the rows of coordinates of $u_{i_1}, \ldots, u_{i_8}$. We then have, in particular,

 \begin{equation*}
\begin{array}{lcc}
\text{det}[u_1, \ldots, u_7, u_8]&{\!\!\!=\!\!\!}&2\\
\text{det}[u_1, \ldots, u_7, u_9]&{\!\!\!=\!\!\!}&-3\\
\text{det}[u_1, \ldots, u_7, u_{10}]&{\!\!\!=\!\!\!}&-1\\
\text{det}[u_2, \ldots, u_8, u_9]&{\!\!\!=\!\!\!}&1.
\end{array}
 \end{equation*}

Hence, the first two combinations are bases of respectively, index 2 and 3 subgroups, the other two combinations are full bases.

We shall now introduce a Riemannian metric on $S_K$ such that the length of $u_1, \ldots , u_{10}$ are the first lines in the length spectrum, i.e.\ such that $u_1, \ldots , u_{10}$ are shortest curves in their homotopy classes, the lengths are increasing,
 \begin{equation*}
\ell(u_1) < \ell(u_2) < \dots < \ell(u_{10}),
 \end{equation*}
and for any other homotopy class the length is greater than $ \ell(u_{10})$. Since $\ell(u_1), \ldots, \ell(u_{10})$ are also the first lines in the homological length spectrum and since $u_1, \ldots, u_{7}$ are extendable to a basis the search of a minimal basis with respect to \textbf{Definition \ref{def:minII}} successively selects $u_1, \ldots, u_{7}$, then rejects $u_8, u_9$ and selects as last element $u_{10}$. However, the result is not minimal as $\{u_2, \ldots, u_8, u_9\}$ is another basis and $\ell(u_9)<\ell(u_{10})$.

For the metric we proceed as earlier by first assigning lengths to the edges of $\mathcal{K}$. Let $\iota_k = 1 +k/200$, for $k=1,\ldots,10$. Any $u_k$ consist of $m_k$ edges with $m_k \in \{2,3,4\}$ and we let $\iota_k/m_k$ be the lengths of these edges. Hence, for each $k=1,\ldots,10$, all edges of $u_k$ have the same length and $\ell(u_k) = \iota_k$. 

Inspecting \textit{Figure \ref{fig:G4}} we see that any closed edge path on $\mathcal{K}$ (with no edge followed by its inverse) that is different from $u_1, \ldots, u_{10}$ has at least three edges, and there are only two such loops, say $v$, $v'$ that consist of exactly three edges. Both of them have their edges on $ \alpha_4$, $\mu$ and $\delta$ and are of length $\ell(v) = \ell(v') > \frac{1}{3} + \frac{1}{2} + \frac{1}{4}$. Any further loop $w$ on $\mathcal{K}$ has at least two edges of lengths $>  \frac{1}{3}$ and satisfies therefore $\ell(w) > \ell(v)$.

The required length spectrum conditions are thus satisfied for the graph $\mathcal{K}$ and we can proceed as in \textbf{Examples 2} and \textbf{3} by first defining a flat metric on a thin ribbon graph around $\mathcal{K}$ in which the $u_k$ are geodesics of length $\iota_k$ and then extend the metric with variable curvature to all of $S_K$ such that still $\ell(u_1), \ldots , \ell(u_{10})$ are the first lines in the length spectrum.

\section{Closing remarks}\label{sec:Clos}

We add a few remarks and observations.

\begin{rem}[Hyperbolic examples]\label{rem:HypExpl}  The construction of a hyperbolic ribbon graph around graph $\mathcal{G}$ that lead to \textbf{Example 3} seems to be new. It can be generalised to any finite graph whose vertices have valencies $\geq 3$. The only thing different is that for a vertex of valency $v$, instead of taking a right angled hyperbolic octagon with four thin arms as in \textit{Figure \ref{fig:octagons}} one has to use an analogous right angled $2v$-gon with $v$ thin arms. The crowns that must be attached along the boundaries to get closed surfaces may become arbitrarily large but the lengths of the geodesics $a_k$, $b_k$ (\textit{Figure \ref{fig:crown}}) are bounded remaining close to the value $2 \arccosh(2) = 2.633 \dots$ The latter arises as a limit when the parameter $w$ goes to infinity and at the same time the parameter $\tau$ goes to 0. Hence, the arguments used in \textbf{Example 3} can be carried out in general.

In this way one can show, in particular, that \textit{there also exist hyperbolic surfaces without globally minimal homology bases}.
\end{rem}

\begin{rem}[Sufficient condition for global minimum]\label{rem:SufCon} \emph{If successive minima procedure I yields a basis then this basis is globally minimal.}

\bigskip

This is part of the next slightly more general lemma. The condition is not  necessary, though, as is detailed in \textbf{Remark \ref{rem:24Rev}}.
\end{rem}

\begin{lem}\label{lem:MinProcI} Let $\gamma_1, \dots, \gamma_{2g} \in H_1(S;\K)$ be obtained by successive minima procedure I. Then for any sequence of linearly independent $\delta_1, \dots \delta_{2g} \in H_1(S;\K)$ ordered by increasing length we have

 \begin{equation*}
\ell(\gamma_k) \leq \ell(\delta_k), \quad k = 1, \dots, {2g}.
 \end{equation*}
 
\end{lem}
\begin{proof} 
Since $\gamma_1$ is a systole the inequality holds for $k=1$. Now assume on the contrary that  there is $j \in \{2, \dots, 2g\}$ with $\ell(\delta_j) < \ell(\gamma_j)$. Under this assumption we must have (with $\text{span}$ over $\K$)
 \begin{equation*}
\delta_i \in \text{span}\{\gamma_1, \dots, \gamma_{j-1}\}, \; \text{for} \; i=1, \dots, j, \tag{*}
 \end{equation*}
for if some $\delta_i$ were not linear combination of $\gamma_1, \dots, \gamma_{j-1}$, then by the minimal procedure and our assumption $\ell(\delta_i) \geq \ell(\gamma_j) > \ell(\delta_j)$, where $i \leq j$, in contradiction to the increasing lengths of $\delta_1, \dots, \delta_{2g}$. Hence, the assumption $\ell(\delta_j) < \ell(\gamma_j)$ implies (*), but (*) cannot hold because $\delta_1, \dots, \delta_j$  are linearly independent.
\end{proof}

\begin{rem}[Necessary condition for a global minimum]\label{rem:StrHom} \emph{If a globally minimal basis exists then successive minima procedure II yields it.} 
\end{rem}

This is an immediate consequence of the definition of the procedure.

\begin{rem}[Local minimum]\label{rem:LocMin} \emph{Successive minima II always yields a local minimum.} This too is an immediate consequence of the definition of the procedure.
\end{rem}

\begin{rem}[Example 2 revisited]\label{rem:24Rev} Graph $\mathcal{G}$ on surface $S_G$ in \textbf{Example  2} contains a closed edge path  $\eta_G$ of edge length 4 that is homotopic on $S_G$ to the curve $\alpha_2$ from the canonical homology basis $\{\alpha, \beta, \alpha_2, \beta_2\}$  (see \textit{Figures \ref{fig:configurationsGH}} and \textit{\ref{fig:edgepaths}}).
It follows that the cycles $\alpha, \beta, \gamma, \eta_G$ induce a basis of $H_1(S_G;\Z)$. Slightly perturbing the lengths we can achieve that $\alpha, \beta, \gamma, \delta_G, \eta_G$ have the unique five first lengths in the length spectrum of $S_G$. The sequence $\alpha, \beta, \gamma, \eta_G$ then forms a globally minimal basis of $H_1(S_G;\Z)$. But successive minima procedure I finds as fourth curve $\delta_G$. A similar remark holds for graph $\mathcal{H}$ on $S_H$ with a closed edge path $\eta_H$ of edge length 3. 

As a conclusion we have that  \emph{an existing global minimum is not always detected by successive minima procedure I.}

\end{rem}

\begin{figure}[!h]
\SetLabels
(.08*.35) $\alpha$\\
(.16*.28) $\beta$\\
(.39*.20) $\gamma$\\
(.36*.70) $\eta_G$\\
(.56*.35) $\,\alpha$\\
(.64*.28) $\,\beta$\\
(.87*.20) $\,\,\gamma$\\
(.83*.70) $\eta_H$\\
(.27*.02) $S_G$\\
(.74*.02) $S_H$\\
\endSetLabels
\AffixLabels{%
\centerline{%
\includegraphics[width=14cm]{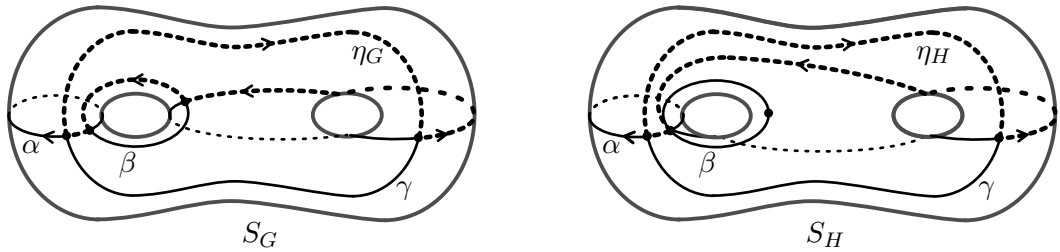}
}}
\caption{Edge paths of length 4 and 3 (thick dashed lines) on the embedded graphs $\mathcal{G}$ and $\mathcal{H}$. }
\label{fig:edgepaths}
\end{figure}

\begin{rem}[Nearly canonical bases in genus 2]\label{rem:nearcang2} We may call a homology basis on a genus $g$ surface $S$ \emph{nearly canonical} if it is induced by simple closed curves that pairwise intersect at most once. In \textbf{Example 1} we gave a counterexample in genus 3 to the question whether the complement of a nearly canonical basis has to be connected. Here we add that in genus 2 a counterexample is not possible.

A way to see this (the only way we know of) is to look at all configurations, up to equivalence, of nearly canonical bases. It turns out (we state this here without proof) that in genus 2 there are only four: the canonical basis and the three configurations in \textit{Figure \ref{fig:non-canonical}}.

\begin{figure}[!t]
\SetLabels
(.13*.72) $\,\,a_1$\\
(.04*.67) $\,a_2$\\
(.18*.79) $\,\,a_3$\\
(.24*.62) $\,a_4$\\
(.16*.49) $\,a_1$\\
(.16*.08) $\,a_1$\\
(.08*.43) $\,a_4$\\
(.24*.14) $a_4$\\
(.05*.29) $\,a_2$\\
(.27*.29) $a_2$\\
(.09*.14) $a_3$\\
(.24*.43) $a_3$\\
(.47*.72) $\,\,a_1$\\
(.38*.68) $\,a_2$\\
(.53*.79) $a_3$\\
(.57*.66) $\,a_4$\\
(.50*.49) $a_1$\\
(.50*.09) $a_1$\\
(.43*.47) $a_2$\\
(.56*.11) $\,a_2$\\
(.38*.35) $\,\,\,\,a_3$\\
(.43*.11) $\,a_3$\\
(.56*.46) $a_3$\\
(.60*.23) $\,a_3$\\
(.39*.22) $a_4$\\
(.60*.35) $\,a_4$\\
(.79*.79) $\,a_1$\\
(.72*.68) $\,a_2$\\
(.86*.69) $a_3$\\
(.91*.66) $\,a_4$\\
(.84*.49) $a_1$\\
(.84*.09) $a_1$\\
(.71*.40) $\,\,a_1$\\
(.95*.18) $\,\,\,a_1$\\
(.77*.48) $a_2$\\
(.90*.09) $a_2$\\
(.69*.29) $a_3$\\
(.77*.10) $a_3$\\
(.98*.28) $a_3$\\
(.89*.48) $\,\,a_3$\\
(.72*.17) $a_4$\\
(.95*.40) $\,\,a_4$\\
\endSetLabels
\AffixLabels{%
\centerline{%
\includegraphics[height=8cm,width=16cm]{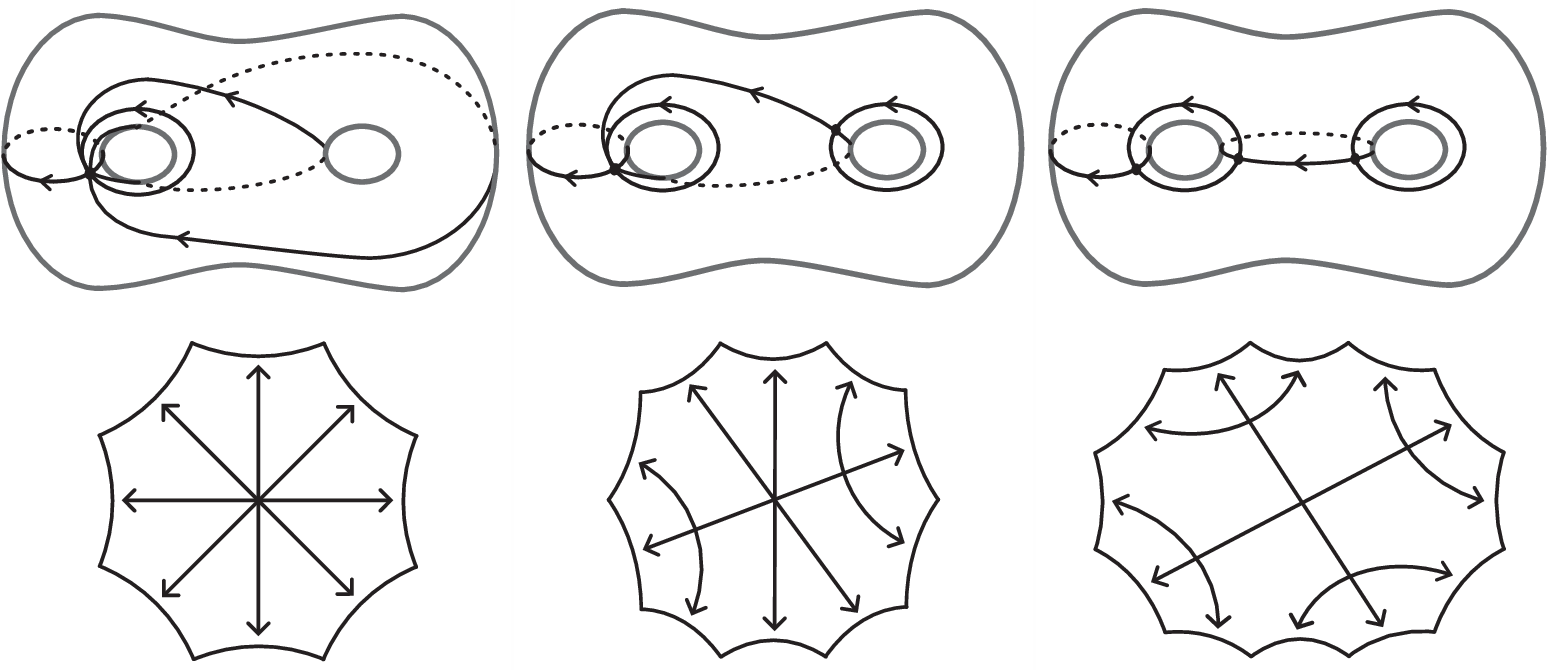}
}}
\caption{The three non-canonical homology bases with pairwise at most one intersection. }
\label{fig:non-canonical}
\end{figure}
For all of them the complement is a connected fundamental domain.

\end{rem}

\begin{rem}[Straight generators]\label{rem:Straight} We recall that a closed curve $c$ on $S$ is called \emph{straight} if for any pair of points $p$, $q$ on $c$ the distance from $p$ to $q$ on $S$ is equal to the length of the shortest arc on $c$ from $p$ to $q$.

The observation in \cite[sect.\ 5.1]{gr} mentioned in the Introduction can be formulated as follows: If $L$ is a subgroup of $H_1(S;\K)$ and $c$ a shortest closed curve on $S$ whose homology class $[c]$ is not in $L$ then $c$ is simple and straight. The reason for this property is that otherwise $c$ may be written as sum of shorter curves $c = c_1 + c_2$ whose homology classes thus both lie in $L$ which is not possible because $[c] \notin L$.  As a consequence successive minima I produces straight curves. The procedure can be continued until the resulting curves generate all of $H_1(S;\K)$. Hence, \emph{$H_1(S;\K)$ can always be generated by straight closed curves.}

The question whether there always exist straight homology \emph{bases}, however, seems to be open.
\end{rem}

\section*{Acknowledgment}

The authors would like to thank Sebastian Baader for helpful discussions and remarks and the referee of the article for his/her helpful comments.

\bigskip

\noindent Peter Buser \\
\noindent Department of Mathematics, Ecole Polytechnique F\'ed\'erale de Lausanne\\
\noindent Station 8, 1015 Lausanne, Switzerland\\
\noindent e-mail: \textit{peter.buser@epfl.ch}\\
\\
\\
\noindent Eran Makover\\
\noindent Department of Mathematics, Central Connecticut State University\\
\noindent 1615 Stanley Street, New Britain, CT 06050, USA\\
\noindent e-mail: \textit{makovere@ccsu.edu}\\
\\
\\
\noindent Bjoern Muetzel \\
\noindent Department of Mathematics, Eckerd College \\
\noindent 4200 54th avenue South, St. Petersburg, FL 33711, USA\\
\noindent e-mail: \textit{bjorn.mutzel@gmail.com}

\end{document}